\documentclass[11pt]{amsproc}
\usepackage{amsmath,amsthm,amssymb,url,hyperref}
\usepackage[all]{xy}

\begin{document}

\newtheorem*{theorems}{Theorem}
\newtheorem{theorem}{Theorem}[section]
\newtheorem{lemma}[theorem]{Lemma}
\newtheorem{proposition}[theorem]{Proposition}
\newtheorem{corollary}[theorem]{Corollary}
\newtheorem{bigtheorem}{Theorem}

\theoremstyle{definition}
\newtheorem{definition}[theorem]{Definition}
\newtheorem{example}[theorem]{Example}
\newtheorem{formula}[theorem]{Formula}
\newtheorem{nothing}[theorem]{}

\theoremstyle{remark}
\newtheorem{remark}[theorem]{Remark}

\renewcommand{\arraystretch}{1.2}

\newcommand{\de}{\partial}
\newcommand{\desude}[2]{{\dfrac{\de #1}{\de #2}}}
\newcommand{\mapor}[1]{{\stackrel{#1}{\longrightarrow}}}
\newcommand{\ormap}[1]{{\stackrel{#1}{\longleftarrow}}}
\newcommand{\mapver}[1]{\Big\downarrow\vcenter{\rlap{$\scriptstyle#1$}}}

\newcommand{\binfty}{\boldsymbol{\infty}}
\newcommand{\bi}{\boldsymbol{i}}
\newcommand{\bl}{\boldsymbol{l}}

\renewcommand{\bar}{\overline}
\renewcommand{\Hat}[1]{\widehat{#1}}

\newcommand{\sA}{\mathcal{A}}
\newcommand{\Oh}{\mathcal{O}}
\newcommand{\sF}{\mathcal{F}}
\newcommand{\sH}{\mathcal{H}}
\newcommand{\sL}{\mathcal{L}}
\newcommand{\sB}{\mathcal{B}}
\newcommand{\sY}{\mathcal{Y}}
\newcommand{\g}{\mathfrak{g}}

\newcommand{\K}{\mathbb{K}}
\newcommand{\Proj}{\mathbb{P}}

\newcommand{\DER}{{{\mathcal D}er}}
\newcommand{\Id}{\operatorname{Id}}

\newcommand{\Spec}{\operatorname{Spec}}
\newcommand{\ad}{\operatorname{ad}}
\newcommand{\MC}{\operatorname{MC}}
\newcommand{\Def}{\operatorname{Def}}
\newcommand{\Hom}{\operatorname{Hom}}
\newcommand{\End}{\operatorname{End}}
\newcommand{\image}{\operatorname{Im}}
\newcommand{\Der}{\operatorname{Der}}
\newcommand{\Mor}{\operatorname{Mor}}
\newcommand{\Cone}{\operatorname{Cone}}
\newcommand{\Aut}{\operatorname{Aut}}
\newcommand{\coker}{\operatorname{Coker}}
\newcommand{\tot}{\operatorname{Tot}}

\newcommand{\Art}{\mathbf{Art}}
\newcommand{\Set}{\mathbf{Set}}

\newcommand{\Grass}{\operatorname{Grass}}
\newcommand{\Flag}{\operatorname{Flag}}
\newcommand{\Hoch}{\operatorname{Hoch}}
\newcommand{\contr}{{\mspace{1mu}\lrcorner\mspace{1.5mu}}}

\newenvironment{acknowledgement}{\par\addvspace{17pt}\small\rm
\trivlist\item[\hskip\labelsep{\it Acknowledgement.}]}
{\endtrivlist\addvspace{6pt}}


\title{An algebraic proof of Bogomolov-Tian-Todorov theorem}
\author{Donatella Iacono}
\address{\newline Institut f\"ur Mathematik,\hfill\newline Johannes
Gutenberg-Universit\"at,
\hfill\newline Staudingerweg 9,
D 55128 Mainz Germany.}
\email{iacono@uni-mainz.de}

\author{Marco Manetti}
\address{\newline Dipartimento di Matematica \lq\lq Guido
Castelnuovo\rq\rq,\hfill\newline
Sapienza Universit\`a di Roma, \hfill\newline
P.le Aldo Moro 5,
I-00185 Roma Italy.}
\email{manetti@mat.uniroma1.it}
\urladdr{www.mat.uniroma1.it/people/manetti/}

\date{}

\begin{abstract}
We give a completely algebraic proof of the Bogomolov-Tian-Todorov theorem. More
precisely, we shall prove that if $X$ is a smooth projective
variety with trivial canonical bundle defined over an
algebraically closed field of  characteristic $0$, then the
$L_{\infty}$-algebra governing infinitesimal deformations of $X$
is quasi-isomorphic to an abelian differential graded Lie algebra.
\end{abstract}

\maketitle

\tableofcontents

\section*{Introduction}

Let $X$ be a smooth projective variety over an algebraically
closed field $\K$ of characteristic 0, with  tangent sheaf
$\Theta_X$. Given an affine open cover $\mathcal{U}=\{U_i\}$ of
$X$, we can consider the \v{C}ech  complex
$\check{C}(\mathcal{U},\Theta_X)$. By classical deformation
theory \cite{Kobook,Sernesi}, the group $H^1(\mathcal{U},\Theta_X)$ classifies first
order deformations of $X$, while $H^2(\mathcal{U},\Theta_X)$ is an
obstruction space for $X$.

Moreover, as a consequence of the results contained  in
\cite{hinich,hinich-schechtman-a,hinich-schechtman-b,scla},  there exists a
canonical sequence of higher brackets on
$\check{C}(\mathcal{U},\Theta_X)$, defining an
$L_{\infty}$ structure and
governing deformations of $X$ over local Artinian $\K$-algebras,
via Maurer-Cartan equation. When $\K=\mathbb{C}$, such
$L_{\infty}$ structure is  canonically quasi-isomorphic to the Kodaira-Spencer
differential graded Lie algebra of $X$, whose Maurer-Cartan equation corresponds
to the integrability condition of almost complex structures \cite{GoMil2}.

Assume now that $X$ has trivial canonical bundle,
the well known Bogomolov-Tian-Todorov (BTT) theorem states that
$X$ has unobstructed deformations. This was first proved
by Bogolomov in \cite{bogomolov}
in the particular case of complex hamiltonian manifolds;
then, Tian \cite{Tian} and Todorov \cite{Todorov} proved  independently
the theorem for  compact K\"{a}hler manifolds
with trivial canonical bundle. Their proofs are transcendental
and make a deep use of the underlying differentiable structure
as well of the $\de\overline{\de}$-Lemma.

More algebraic proofs of BTT theorem, based on $T^1$-lifting theorem
and degeneration of the Hodge spectral sequence,  were  given in
\cite{zivran} for $\K=\mathbb{C}$  and in \cite{Kaw1,FM2}
for any $\K$ as above.

For $\K=\mathbb{C}$, the BTT theorem is also a consequence of the
stronger result \cite{GoMil2,manRENDICONTi}
that the Kodaira-Spencer differential graded Lie algebra of
$X$ is quasi-abelian, i.e.,
quasi-isomorphic to an abelian differential graded Lie algebra.
This result, also proved with transcendental methods, is important for
applications to Mirror symmetry \cite{BK}.

The main result of this paper is to give a purely algebraic proof of the
quasi-abelianity of the $L_{\infty}$-algebra
$\check{C}(\mathcal{U},\Theta_X)$ when $X$ is projective with trivial
canonical bundle. This is achieved using
the degeneration of the Hodge-de Rham
spectral sequence, proved algebraically by
Faltings,  Deligne and Illusie  \cite{faltings,DI},  and the $L_{\infty}$
description of the period map \cite{Generalized fi-ma}.

More precisely, we prove the following theorems.

\begin{theorems}[A]
Let  $X$ be a smooth projective variety  of dimension $n$ defined
over an algebraically closed field of characteristic 0. If the
contraction map
\[ H^*(\Theta_X)\xrightarrow{\bi}\Hom^*(H^*(\Omega^n_X),H^*(\Omega^{n-1}_X))\]
is injective, then for every affine open cover $\mathcal{U}$ of
$X$, the $L_\infty$-algebra $\check{C}(\mathcal{U},\Theta_X)$ is
quasi-abelian.
\end{theorems}

\begin{theorems}[B]
Let $X$ be   a smooth projective variety defined over an
algebraically closed field of characteristic 0. If the canonical
bundle of $X$ is trivial or torsion, then
the $L_\infty$-algebra $\check{C}(\mathcal{U},\Theta_X)$ is
quasi-abelian.
\end{theorems}

The paper goes as follows: the first section is intended for the
non expert reader and is devoted to recall the basic notions of
differential graded Lie algebras, $L_{\infty}$-algebras and their
role in  deformation theory.

In Section~\ref{sec.scla},  we review the construction of the
Thom-Whitney complex associated with a semicosimplicial complex.

In Section \ref{sez.formality map-cone},
 following \cite{scla},
we introduce  semicosimplicial differential graded Lie algebras,
the associated Thom-Whitney DGLAs  and  the  $L_{\infty}$
structure on the associated total complexes. We also investigate
some properties of   mapping cones associated with a morphism of
DGLAs.

In  Sections  \ref{sez.cartan homotopy}, we collect some technical
results about Cartan homotopies and contractions.

In Section~\ref{sezione.deformazioni varieta},  we give the
definition of the deformation functor $H^1_{\rm sc}(
\exp\g^{\Delta})$, associated with a semicosimplicial Lie algebra
$\g^{\Delta}$, introduced essentially in \cite{hinich,Pridham} and
described in more detailed way in \cite{scla}. Moreover, following
\cite{scla}, we prove, in a complete algebraic way, that the
infinitesimal deformations  of a smooth variety $X$, defined over
a field of characteristic 0, are controlled by the
$L_\infty$-algebra $\check{C}(\mathcal{U},\Theta_X)$, where
$\mathcal{U}$ is an open affine cover of $X$.

Section \ref{sez.Proof main theorem} is devoted to the algebraic
proof of the previous Theorems A and B together with some
applications to deformation theory.

\begin{acknowledgement} The first draft of this paper was
written during the workshop on \lq\lq Algebraic and Geometric
Deformation Spaces\rq\rq,  Bonn August 11-15, 2008. Both authors
thank Max Planck Institute  and Hausdorff Center for Mathematics
for financial support and warm hospitality. The first author wish
to thank the Mathematical Department ``G. Castelnuovo'' of
Sapienza Universit\`a di Roma for the hospitality. Thanks also to
D. Fiorenza and E. Martinengo for useful discussions on the
subject of this paper. We are grateful to  the referee for
improvements in the exposition  of the paper.
\end{acknowledgement}

\section{Review of DGLAs and $L_{\infty}$-algebras}

Let $\K$ be a fixed algebraically closed field of characteristic
zero.  A \emph{differential graded vector space} is a pair
$(V,d)$, where $V=\oplus_i V^i$ is a $\mathbb{Z}$-graded vector
space and $d\colon V^i\to V^{i+1}$ is a differential of degree
$+1$. For every  integer $n$, we define a new differential graded
vector space $V[n]$ by setting
\[ V[n]^i=V^{n+i},\qquad  d_{V[n]}=(-1)^nd_V.\]

A \emph{differential graded Lie algebra} (DGLA for short) is
the data of a differential graded vector space $(L,d)$ together
with a  bilinear map $[-,-]\colon L\times L\to L$ (called bracket)
of degree 0 such that:\begin{enumerate}

\item (graded  skewsymmetry) $[a,b]=-(-1)^{\deg(a)\deg(b)}[b,a]$.

\item (graded Jacobi identity) $
[a,[b,c]]=[[a,b],c]+(-1)^{\deg(a)\deg(b)}[b,[a,c]]$.

\item (graded Leibniz rule) $d[a,b]=[da,b]+(-1)^{\deg(a)}[a,db]$.
\end{enumerate}

In particular,
the Leibniz rule implies that the bracket of a DGLA
$L$ induces a structure of graded Lie algebra on its cohomology
$H^*(L)=\oplus_iH^i(L)$.

\begin{example}
Let $(V,d_V)$ be a differential graded vector space and $\operatorname{Hom}^i(V,V)$
the space of morphisms  $V\to V$ of degree $i$. Then,
$\operatorname{Hom}^*(V,V)=\bigoplus_i\operatorname{Hom}^i(V,V)$
is a DGLA with bracket
\[
[f,g]=fg-(-1)^{\deg(f)\deg(g)}gf,
\]
and differential $d$ given by
\[
d(f)=[d_V,f]=d_Vf-(-1)^{\deg(f)}fd_V.
\]
For later use, we point out that there exists
a natural isomorphism
\[ H^*(\Hom^*(V,V))\mapor{\simeq}\Hom^*(H^*(V),H^*(V)).\]
\end{example}

A morphism of differential graded Lie algebras $\varphi\colon L \to M$ is  a linear map
that preserves degrees and commutes with brackets and
differentials. A \emph{quasi-isomorphism} of DGLAs is a morphism
that induces an isomorphism in cohomology. Two DGLAs $L$ and $M$ are said to be
\emph{quasi-isomorphic} if they are equivalent under the
equivalence relation generated by: $L\sim M$ if there exists
a quasi-isomorphism $\phi: L\to M$.\par

Next, denote by $\Set$  the category of sets (in a fixed universe)
and by $\Art=\Art_\K$   the category of local Artinian
$\K$-algebras with residue field $\K$. Unless otherwise specified,
for every  objects  $A\in \mathbf{Art}$, we denote by
$\mathfrak{m}_A$ its maximal ideal. Given a DGLA $L$,   we define
the Maurer-Cartan functor $\MC_L\colon \mathbf{Art}\to
\mathbf{Set}$ by setting \cite{ManettiDGLA}:
\[
\MC_L(A)=\left\{ x \in L^1\otimes \mathfrak{m}_A \mid dx+ \displaystyle\frac{1}{2}
[x,x]=0 \right\},
\]
where the DGLA structure on $L \otimes \mathfrak{m}_A$ is the natural
extension of the DGLA structure on $L$.
The gauge action $\ast:\exp(L^0 \otimes \mathfrak{m}_A)\times
\MC_L(A)\longrightarrow {\MC}_L(A)$ may be defined by the explicit formula
\[
e^a \ast x:=x+\sum_{n\geq 0} \frac{ [a,-]^n}{(n+1)!}([a,x]-da).
\]
The deformation functor $\Def_L:\Art \longrightarrow \Set$
associated to a DGLA $L$ is:
\[
\Def_L(A)=\frac{\MC_L(A)}{\text{gauge}}=\frac{\{ x \in L^1\otimes \mathfrak{m}_A \ |\ dx+
\displaystyle\frac{1}{2} [x,x]=0 \}}{\exp(L^0\otimes \mathfrak{m}_A )  }.
\]

\begin{remark}
Every morphism of DGLAs induces a natural transformation of the
associated deformation functors. A basic result asserts that if
$L$ and $M$ are quasi-isomorphic DGLAs, then the associated
functor $\Def_L$ and $\Def_M$ are isomorphic
\cite{SchSta,GoMil1,GoMil2}, \cite[Corollary~3.2]{ManettiDGLA},
\cite[Corollary~5.52]{manRENDICONTi}.
\end{remark}

Next, we briefly recall the definition of an $L_{\infty}$
structures on a graded vector space $V$. For a more detailed
description of such structures we refer to
\cite{SchSta,LadaStas,LadaMarkl,EDF,fukaya,K,getzler,cone} and
\cite[Chapter~IX]{manRENDICONTi}.

Let $V$ be a graded vector space:  we
denote by $\bigodot^n V$ its graded symmetric $n$-th power.
Given $v_1,\ldots,v_n$ homogeneous elements of $V$,
for every permutation $\sigma$ we have
\[
v_1 \odot \cdots \odot  v_n=\epsilon (\sigma;
v_1,\ldots , v_n)\ v_{\sigma(1)} \odot \cdots \odot
v_{\sigma(n)},
\]
where $\epsilon (\sigma; v_1,\ldots , v_n)$ is the  Koszul sign.
When the sequence  $v_1, \ldots , v_n$ is clear from the context, we simply
write $\epsilon(\sigma)$ instead of $\epsilon (\sigma;v_1,\ldots,v_n)$.

\begin{definition} Denote by $\Sigma_n$ the
group of permutations of the set $\{1,2,\ldots,n\}$.
The set of \emph{unshuffles} of type $(p,n-p)$ is the subset
$S(p,n-p)\subset\Sigma_{n}$ of permutations $\sigma$ such that
$\sigma(1)<\sigma(2) < \cdots < \sigma(p)$ and $\sigma(p+1)
<\sigma(p+2) < \cdots < \sigma(n)$.
\end{definition}

\begin{definition}
An $L_{\infty}$ structure  on a graded vector space $V$ is a
sequence $\{q_k\}_{k \geq 1}$ of linear maps $q_k \in \Hom^1(\odot ^k
(V[1]),V[1])$ such that the map
$$
Q:\bigoplus_{n \geq 1} \bigodot^n V[1] \to \bigoplus_{n \geq 1}
\bigodot^n V[1],
$$
defined as
$$
Q(v_1 \odot \cdots \odot  v_n)=\sum_{k=1}^n \sum_{\sigma \in
S(k,n-k)} \epsilon(\sigma)q_k( v_{\sigma(1)} \odot \cdots \odot
v_{\sigma(k)})\odot v_{\sigma(k+1)} \odot \cdots \odot
v_{\sigma(n)},
$$
is a codifferential on the reduced symmetric graded coalgebra $\bigoplus_{n \geq 1}
\bigodot^n V[1]$: in other words,  the datum $(V,q_1,q_2,q_3,
\ldots )$  is called an $L_{\infty}$-algebra if $QQ=0$.
\end{definition}

If $(V,q_1,q_2,q_3, \ldots )$  is an $L_{\infty}$-algebra, then
$q_1q_1=0$ and therefore $(V[1],q_1)$ is a differential graded
vector space. The relation between DGLAs and $L_{\infty}$-algebras
is given by the following example.

\begin{example}[\cite{Qui}]\label{oss DGLA is L infinito}
Let $(L,d,[\, , \, ])$ be a differential graded Lie algebra and define:
\[
q_1=-d: L[1] \to L[1],\]
\[q_2 \in \Hom^1(\odot ^2 (L[1]),L[1]),\qquad
q_2(v\odot w)=(-1)^{\deg(v)}[v,w],
\]
and $q_k=0$ for every $k \geq 3$. Then $(L,q_1,q_2,0,\ldots)$ is
an $L_{\infty}$-algebra.
\end{example}

An \emph{$L_{\infty}$-morphism} $(V,q_1,q_2, \ldots
)\to (W,p_1,p_2, \ldots )$ of $L_{\infty}$-algebras is a
sequence $f_{\infty}=\{f_n\}$ of degree zero linear maps
\[
f_n: \bigodot^n V[1] \to W[1], \qquad n\geq 1,
\]
such that the (unique) morphism of graded coalgebras
\[
F: \bigoplus_{n \geq 1} \bigodot^n V[1] \to \bigoplus_{n \geq 1}
\bigodot^n W[1],
\]
lifting $\sum_n f_n :\bigoplus_{n\geq 1} \bigodot^n V[1] \to
W[1]$, commutes with the codifferentials.
This condition implies that the \emph{linear part} $f_1:V[1]\to W[1] $ of an
$L_{\infty}$-morphism $f_{\infty}: (V,q_1,q_2, \ldots )\to
(W,p_1,p_2, \ldots )$ satisfies the condition $f_1 \circ q_1
=p_1 \circ f_1$,  and therefore $f_1$ is a map of differential complexes
$(V[1],q_1) \to (W[1], p_1)$.\par

An $L_{\infty}$-morphism $f_{\infty}=\{f_n\}$ is said to be linear
if $f_n=0$, for every $n \geq 2$. Notice that an
$L_{\infty}$-morphism between two DGLAs is linear if and only if
it is a morphism of differential graded Lie algebras.

A \emph{quasi-isomorphism} of $L_{\infty}$-algebra is an
$L_{\infty}$-morphism, whose linear part is a quasi-isomorphism of
complexes. Two $L_{\infty}$-algebras are \emph{quasi-isomorphic}
if they are equivalent under the
equivalence relation generated by the relation: $L\sim M$ if there exists
a quasi-isomorphism $\phi: L\to M$.

\begin{definition}\label{def.quasiabelian}
An $L_{\infty}$-algebra $(V,q_1,q_2,\ldots)$ is called \emph{abelian} if
$q_i=0$ for every $i\ge 2$.
An $L_{\infty}$-algebra  is called
\emph{quasi-abelian} if it is quasi-isomorphic to
an abelian $L_{\infty}$-algebra.
\end{definition}

Given an $L_\infty$-algebra $(V,q_1,q_2,\ldots)$ and a
differential graded commutative algebra $A$, there exists a
natural $L_{\infty}$ structure on the tensor product $V\otimes A$.
The \emph{Maurer-Cartan functor $\MC_V $ associated with the
$L_{\infty}$-algebra $V$} is the functor \cite{SchSta,fukaya,K}:
$$
\MC_V  : \Art \to \Set
$$
\[
\MC_V (A)=\left\{ \gamma\in V[1]^0\otimes \mathfrak{m}_A
\;\strut\left\vert\; \sum_{j\ge1}\frac{q_j(\gamma^{\odot
j})}{j!}=0\right.\right\}.
\]

Two elements $x$ and $y \in\MC_V (A)$ are \emph{homotopy
equivalent} if there exists $g(s)\in \MC_{V\otimes\K[s,ds]}(A)$ such that
$g(0)=x$ and $g(1)=y$. Then, the \emph{deformation functor}
$\Def_V$ associated with the $L_\infty$-algebra $V$ is
\[\Def_V  : \Art \to \Set,\qquad
\Def_V(A)=\frac{\MC_ V(A)}{\text{homotopy}}.
\]

\begin{example}
Given an abelian $L_\infty$-algebra $(V,q_1,0,0,\ldots)$, for every $A\in\Art$ there exists a  canonical isomorphism
\[ \Def_V(A)=H^0(V[1],q_1)\otimes\mathfrak{m}_A.\]
\end{example}

\begin{remark}
If $L$ is a DGLA, then the   deformation functor associated with
$L$, viewed as an   $L_\infty$-algebra, is isomorphic to the
previous one (Maurer-Cartan modulo gauge equivalence)
\cite{SchSta,EDF,fukaya,K}.
\end{remark}

\begin{remark}
As for the DGLAs, every morphism of $L_{\infty}$-algebras induces
a natural transformation of  the associated deformation functors.
If two $L_\infty$-algebras are quasi-isomorphic, then there exists
an isomorphism between the associated deformation functors
\cite{fukaya,K}, \cite[Corollary IX.22]{manRENDICONTi}. In
particular, the deformation functor of a quasi-abelian
$L_{\infty}$-algebra is unobstructed.
\end{remark}

\begin{lemma}\label{lem.criterioquasiabelianita}
Let $(V,q_1,q_2,\ldots)$, $(W,r_1,r_2,\ldots)$ be  $L_{\infty}$-algebras
with $W$ quasi-abelian and $f_{\infty}\colon V\to W$  an
$L_{\infty}$-morphism. If $f_1$ is injective in cohomology, then $V$ is quasi-abelian.
\end{lemma}

\begin{proof} According to homotopy classification of  $L_{\infty}$-algebras
(see e.g. \cite{K}), if $H$ is the cohomology of the complex $(W,r_1)$,
there exists an $L_{\infty}$ structure on $H$ and a surjective quasi-isomorphism
of $L_{\infty}$-algebras $p_{\infty}\colon
W\to H$. The $L_{\infty}$ structure on $H$ depends, up to isomorphism,
by the quasi-isomorphism class of $W$ and therefore every bracket on $H$ is trivial.

Replacing $W$ with $H$ and $f_{\infty}$ with $p_{\infty}f_{\infty}$
it is not restrictive to assume $r_i=0$ for every $i$.
There exist a graded vector space $K$ and a morphism of  graded vector
spaces $\beta\colon W[1]\to K[1]$ such that the composition
\[ (V[1],q_1)\mapor{f_1}(W[1],0)\mapor{\beta}(K[1],0)\]
is a quasi-isomorphism of complexes. Then,  $\beta$ is a linear
$L_{\infty}$-morphism and the composition
\[ (V,q_1,q_2,\ldots)\xrightarrow{f_{\infty}}(W,0,0,\ldots)\xrightarrow{\beta}(K,0,0,\ldots)\]
is a quasi-isomorphism of $L_\infty$-algebras.
\end{proof}

\section{The Thom-Whitney complex}
\label{sec.scla}

Let $\mathbf{\Delta}_{\operatorname{mon}}$ be  the category whose
objects are the finite ordinal sets $[n]\!=\!\{0,1,\ldots,n\}$,
$n=0,1,\ldots$, and whose morphisms are order-preserving injective
maps among them. Every morphism in
$\mathbf{\Delta}_{\operatorname{mon}}$, different from the
identity, is a finite  composition of \emph{coface} morphisms:
\[
\partial_{k}\colon [i-1]\to [i],
\qquad \partial_{k}(p)=\begin{cases}p&\text{ if }p<k\\
p+1&\text{ if }k\le p\end{cases},\qquad k=0,\dots,i.
\]
The relations about compositions of them are generated by
\[ \partial_{l}\partial_{k}=
\partial_{k+1}\partial_{l}\, ,\qquad\text{for every }l\leq k.\]

According to \cite{EZ,weibel}, a \emph{semicosimplicial} object in
a category $\mathbf{C}$ is a  covariant functor $A^\Delta\colon
\mathbf{\Delta}_{\operatorname{mon}}\to \mathbf{C}$. Equivalently,
a semicosimplicial  object $A^\Delta$ is a diagram in
$\mathbf{C}$:
 \[
\xymatrix{ {A_0}
\ar@<2pt>[r]\ar@<-2pt>[r] & { A_1}
      \ar@<4pt>[r] \ar[r] \ar@<-4pt>[r] & { A_2}
\ar@<6pt>[r] \ar@<2pt>[r] \ar@<-2pt>[r] \ar@<-6pt>[r]&
\cdots ,}
\]
where each $A_i$ is in $\mathbf{C}$, and, for each $i>0$,
there are $i+1$ morphisms
\[
\partial_{k}\colon {A}_{i-1}\to {A}_{i},
\qquad k=0,\dots,i,
\]
such that $\partial_{l}
\partial_{k}=\partial_{k+1} \partial_{l}$, for any $l\leq k$.\par

Given a  semicosimplicial differential
graded vector space
\[V^{\Delta}:\quad
\xymatrix{ {V_0}
\ar@<2pt>[r]\ar@<-2pt>[r] & { V_1}
      \ar@<4pt>[r] \ar[r] \ar@<-4pt>[r] & { V_2}
\ar@<6pt>[r] \ar@<2pt>[r] \ar@<-2pt>[r] \ar@<-6pt>[r]& \cdots ,}\]
the graded vector space $\bigoplus_{n\ge 0}V_n[-n]$ has two
differentials
\[ d=\sum_{n}(-1)^nd_n,\qquad \text{where}\quad d_n
\text{ is the differential of } V_n,
\]
and
\[
\partial=\sum_{i}(-1)^i\partial_i,\qquad \text{where}
\quad \partial_i\text{ are the coface maps}.\]
More explicitly, if
$v\in V^i_n$, then the degree of $v$ is $i+n$ and
\[ d(v)=(-1)^nd_n(v)\in V^{i+1}_n,\qquad
\partial(v)=\partial_0(v)-\partial_1(v)+\cdots+(-1)^{n+1}
\partial_{n+1}(v)\in V_{n+1}^i.\]
Since $d\partial+\partial d=0$, we define $\tot(V^{\Delta})$ as
the  graded vector space  $\bigoplus_{n\ge 0}V_n[-n]$, endowed
with the differential $d+\partial$.

\begin{example} \label{example. chec Tangen X}

Let $\mathcal{U}=\{U_i\}$ be an affine open cover of a smooth
variety $X$, defined over an algebraically closed field of
characteristic 0;  denote by ${\Theta}_X$ the tangent sheaf  of
$X$. Then,   we can define  the   \v{C}ech semicosimplicial Lie
algebra $\Theta_X(\mathcal{U})$ as the semicosimplicial Lie
algebra
\[
\Theta_X(\mathcal{U}):\quad
\xymatrix{ {\prod_i\Theta_X(U_i)}
\ar@<2pt>[r]\ar@<-2pt>[r] & {
\prod_{i<j}\Theta_X(U_{ij})}
      \ar@<4pt>[r] \ar[r] \ar@<-4pt>[r] &
      {\prod_{i<j<k}\Theta_X(U_{ijk})}
\ar@<6pt>[r] \ar@<2pt>[r] \ar@<-2pt>[r] \ar@<-6pt>[r]& \cdots},
\]
where  the coface maps   $  \displaystyle \partial_{h}\colon
{\prod_{i_0<\cdots <i_{k-1}}\Theta_X(U_{i_0 \cdots  i_{k-1}})}\to
{\prod_{i_0<\cdots <i_k}\Theta_X(U_{i_0 \cdots  i_k})}$
are given by
\[\partial_{h}(x)_{i_0 \ldots i_{k}}={x_{i_0 \ldots
\widehat{i_h} \ldots i_{k}}}_{|U_{i_0 \cdots  i_k}},\qquad
\text{for }h=0,\ldots, k.\]

Since every Lie algebra is, in particular, a differential graded
vector  space (concentrated in degree $0$), it makes sense to
consider the total complex $\tot(\Theta_X(\mathcal{U}))$, which
coincides with the \v{C}ech complex
$\check{C}(\mathcal{U},\Theta_X)$.
\end{example}

\begin{example} \label{example. chec OMEGA X}
Let $\Omega^{\ast}_X$ be the algebraic de Rham complex of a smooth
variety $X$ of dimension $n$:
\[
\Omega_X^{\ast}\colon \qquad 0\to \mathcal{O}_X=
\Omega^0_X\mapor{d}\Omega^1_X\mapor{d}\cdots\mapor{d}\Omega^n_X\to 0.
\]
Given an affine open cover $\mathcal{U}=\{U_i\}$ of $X$,
 we can define a  semicosimplicial differential graded
vector space
\[
\Omega^{\ast}_X(\mathcal{U}):\quad
\xymatrix{ {\prod_i\Omega^{\ast}_X(U_i)}
\ar@<2pt>[r]\ar@<-2pt>[r] & {
\prod_{i<j}\Omega^{\ast}_X(U_{ij})}
      \ar@<4pt>[r] \ar[r] \ar@<-4pt>[r] &
      {\prod_{i<j<k}\Omega^{\ast}_X(U_{ijk})}
\ar@<6pt>[r] \ar@<2pt>[r] \ar@<-2pt>[r] \ar@<-6pt>[r]& \cdots}\; ,
\]
where $\displaystyle \prod_{i_0< \cdots <i_k}
\Omega^{\ast}_X(U_{i_0  \cdots  i_k}) $ is the complex
\[
\prod_{i_0<i_1<\cdots <i_k} \!\!\!
\mathcal{O}_X(U_{i_0\cdots i_k}) \stackrel{d} {\to} \!\!\!
\prod_{i_0<i_1<\cdots <i_k} \!\!\! \Omega^1_X(U_{i_0\cdots
i_k}) \stackrel{d} {\to} \   \cdots \!\!\! \prod_{i_0<\cdots
<i_n} \!\!\! \Omega^n_X (U_{i_0\cdots i_n}).
\]
Here the total complex $\tot( \Omega^{\ast}_X(\mathcal{U}))$ is
the  \v{C}ech complex $\check{C}(\mathcal{U},\Omega^{\ast}_X)$  of
$\Omega^*_X$, with respect to the affine cover $\mathcal{U}$.
\end{example}

\bigskip

Let  $V^{\Delta}$ be a   semicosimplicial differential graded
vector space and
 $(A_{PL})_n$  the differential graded commutative algebra
of polynomial differential forms on the standard $n$-simplex
$\{(t_0,\ldots,t_n)\in \K^{n+1}\mid \sum t_i=1\}$ \cite{FHT}:
\[ (A_{PL})_n=\frac{\K[t_0,\ldots,t_n,dt_0,\ldots,dt_n]}
{(1-\sum t_i,\sum dt_i)}.\] For every $n,m$ the tensor product
$V_n\otimes (A_{PL})_m$ is a differential graded vector space and
then also $\prod_n V_n \otimes (A_{PL})_n$ is a differential
graded vector space.

Denoting by
\[ \delta^{k}\colon (A_{PL})_n\to
(A_{PL})_{n-1},\quad
\delta^{k}(t_i)=\begin{cases}t_i&\text{ if }0\le i<k\\
0&\text{ if }i=k\\
t_{i-1}&\text{ if }k<i\end{cases},\qquad k=0,\dots,n,
\]
the face maps, for every $0\le k\le n$, there are well-defined
morphisms of   differential graded vector spaces
\[
Id\otimes \delta^{k}\colon V_n\otimes (A_{PL})_{n}\to V_n\otimes
(A_{PL})_{n-1},\]
\[\partial_{k}\otimes Id \colon V_{n-1}\otimes
(A_{PL})_{n-1}\to V_{n}\otimes (A_{PL})_{n-1}.
\]
The Thom-Whitney differential graded vector space
$\tot_{TW}(V^\Delta)$ of $V^\Delta$ is the  differential graded
subvector space of $\prod_n V_n \otimes (A_{PL})_n$, whose
elements are the sequences $(x_n)_{n\in\mathbb{N}}$ satisfying the
equations
\[(Id\otimes \delta^{k})x_n=
(\partial_{k}\otimes Id)x_{n-1},\; \text{ for every }\; 0\le k\le
n.
\]


In \cite{whitney}, Whitney noted  that the integration maps
\[ \int_{\Delta^n}\otimes \operatorname{Id}\colon (A_{PL})_{n}\otimes
V_n\to {\mathbb C}[n]\otimes V_n=V_n[n]\] give a quasi-isomorphism
of differential graded vector spaces
\[
I\colon ({\tot}_{TW}(V^\Delta), d_{TW})\to
({\tot}(V^\Delta),d_{\tot}).
\]
Moreover, there exist an explicit injective  quasi-isomorphism of
differential graded vector spaces
\[
E\colon {\tot}(V^\Delta)\to {\tot}_{TW}(V^\Delta)
\]
and an explicit homotopy \[ h\colon {\tot}_{TW}(V^\Delta)\to
{\tot}_{TW}(V^\Delta)[-1],
\]
such that
\[
IE={\rm Id}_{{\tot}(V^\Delta)};\qquad EI-{\rm
Id}_{{\tot}_{TW}(V^\Delta)}=hd_{TW}+d_{TW}h.
\]
Moreover, the morphisms $I,E,h$ are functorial and commute with
morphisms of semicosimplicial differential graded vector spaces.
For more details and explicit description of $E$ and $h$,  we
refer to \cite{dupont1,dupont2,navarro,getzler,scla,chenggetzler}.

\bigskip

Let $V^\Delta$ and $W^\Delta$ be two  semicosimplicial
 graded vector spaces, with degeneracy maps
$\de_{k,n}$  and $\de^{'}_{k,n}$, respectively. Then, we define
the tensor product  $ (V\otimes W)^\Delta$  as the
semicosimplicial graded vector space, such that $(V\otimes
W)^\Delta_n=V_n \otimes W_n$ and the degeneracy maps are defined
on each factor, i.e., $\de_{k,n}=\de^{'}_{k,n}\otimes
\de^{''}_{k,n}:(V\otimes W)^\Delta_{n-1} \to (V\otimes W)^\Delta_n
$.

\begin{remark}\label{remark. tot X tot -- tot}
The Thom-Withney construction is compatible with tensor product,
i.e., there exists a natural transformation
\[
\Phi \colon \tot_{TW} (V^\Delta) \otimes \tot_{TW} (W^\Delta) \to
 \tot_{TW} ((V\otimes W)^\Delta)
\]
Indeed, we have to  prove that $x \otimes y \in
\tot_{TW}((V\otimes W)^\Delta)$, for every pair of  sequences
$x=(x_n)_{n\in\mathbb{N}} \in \tot_{TW}(V^\Delta)$ and
$y=(y_n)_{n\in\mathbb{N}} \in \tot_{TW}(W^\Delta)$.

We have
\[ \partial_{k}\otimes Id(x_n \otimes y_n)=
(\partial_{k}\otimes Id(x_n))\otimes (\partial_{k}\otimes
Id(y_n)),\] and,
\[ Id\otimes\delta^{k}(x_{n+1}\otimes y_{n+1})=
(Id\otimes\delta^{k}(x_{n+1}))\otimes
(Id\otimes\delta^{k}(y_{n+1})).
\]
It is sufficient to observe that the  right parts of the above two
equation are the same for every $k\le n$.

In particular, every bilinear map of semicosimplicial graded
vector spaces $ V^\Delta \times W^\Delta \to Z^\Delta$ induces a
bilinear map $\tot_{TW} (V^\Delta) \times \tot_{TW} (W^\Delta) \to
 \tot_{TW} (Z^\Delta)$.
\end{remark}

\section{Semicosimplicial differential graded Lie algebras and  mapping
cones}
\label{sez.formality map-cone}

Let
 \[\mathfrak{g}^\Delta:\quad
\xymatrix{ {{\mathfrak g}_0} \ar@<2pt>[r]\ar@<-2pt>[r] & {
{\mathfrak g}_1}
      \ar@<4pt>[r] \ar[r] \ar@<-4pt>[r] & { {\mathfrak g}_2}
\ar@<6pt>[r] \ar@<2pt>[r] \ar@<-2pt>[r] \ar@<-6pt>[r]& \cdots ,}
\]
be a \emph{semicosimplicial differential graded Lie algebra}.
Every $\mathfrak{g}_i$ is a DGLA and so, in particular,  a
differential graded vector space, thus we can consider the total
complex $\tot(\mathfrak{g}^\Delta)$.

\begin{example}\label{ex.conisonoscla}
Every morphism $\chi\colon L\to M$ of differential graded Lie
algebras can be interpreted as  the semicosimplicial DGLA
 \[\mathfrak{\chi}^\Delta:\quad
\xymatrix{ {L} \ar@<2pt>[r]\ar@<-2pt>[r] & { M}
      \ar@<4pt>[r] \ar[r] \ar@<-4pt>[r] & 0},\qquad
      \partial_0=\chi,\, \partial_1=0,
\]
and the total complex $\tot(\chi^{\Delta})$ coincides  with the
mapping cone of $\chi$, i.e.,
\[ \tot(\chi^{\Delta})^i=L^i\oplus M^{i-1},\qquad d(l,m)=(dl,\chi(l)-dm).\]

Even in the case of $L$ and $M$ Lie algebras, it is not possible
to define a canonical bracket on the mapping cone, making
$\tot(\chi^{\Delta})$ a DGLA and the projection
$\tot(\chi^{\Delta})\to L$ a morphism of DGLAs. To see this it is
sufficient to consider $L=M$ the Lie subalgebra of $sl(2,\K)$
generated by the matrices
\[ A=\begin{pmatrix}0&1\\ 0&0\end{pmatrix},
\qquad B=\begin{pmatrix}1&0\\ 0&-1\end{pmatrix},\] and $\chi$
equal to the identity. If $\tot(\chi^{\Delta})$ is  a DGLA then,
by functoriality, for every $x\in L$ the subspace generated by
$x,\chi(x)$ is a subalgebra and, for every $\lambda\in\K$, the
linear map $B\mapsto B$, $A\mapsto \lambda A$ is an automorphism
of DGLA. It is an easy exercise to prove that these properties
imply the failure of Jacobi identity.
\end{example}

Even if the complex ${\tot}(\g^{\Delta})$ has no natural DGLA
structure, it can be endowed with a canonical $L_\infty$ structure
by homological perturbation theory \cite{cone,scla}.

Indeed, as in the previous section, we can   consider the
Thom-Whitney construction also for semicosimplicial differential
graded Lie algebras: it is evident that in such case we have
$\tot_{TW}(V^\Delta)$ a differential graded lie algebra. In this
case, the morphisms $I,E,h$ are functorial and commute with
morphisms of semicosimplicial DGLAs. Moreover, these  morphisms
$I,E,h$ can be used to apply the following basic result about
$L_\infty$-algebras, dating back to Kadeishvili's work on the
cohomology of $A_\infty$ algebras \cite{Kad}; see also \cite{HK}.

\begin{theorem}[Homotopy transfer]\label{teo trasfer strutt L infiniti}
Let $(V,q_1,q_2,q_3, \ldots)$ be an $L_\infty$-algebra and $(C,
\delta)$   a differential graded vector space. Assume we have  two
morphisms of complexes
\[
\pi:(V[1],q_1) \to (C[1],\delta_{[1]}), \qquad \imath_1:
(C[1],\delta_{[1]}) \to (V[1],q_1)
\]
and a linear map $h\in \Hom^{-1}(V[1],V[1])$ such that
$hq_1+q_1h=\imath_1\pi-Id$.\par

Then, there exist a canonical $L_\infty$-algebra structure
$(C,\langle\,\rangle_1, \langle\,\rangle_2, \ldots )$ on $C$
extending its differential complex structure, and  a canonical
$L_\infty$-morphism $\imath_\infty\colon (C,\langle\,\rangle_1,
\langle\,\rangle_2, \ldots )\to (V,q_1,q_2,q_3, \ldots)$ extending
$\imath_1$. In particular, if $\imath_1$ is an injective
quasi-isomorphism of complexes, then also $\imath_{\infty}$ is an
injective quasi-isomorphism of $L_{\infty}$-algebras.
\end{theorem}

\begin{proof}
See \cite{cone} and references therein; explicit formulas for the
quasi-isomorphism $\imath_\infty$ and the brackets
$\langle\,\rangle_n^{}$ have been described by Merkulov in
\cite{Merk}; then, it has   been remarked by Kontsevich and
Soibelman in \cite{KoSo,KonSoi} (see also \cite{fukaya}) that
Merkulov's formulas can be nicely written as summations over
rooted trees.
\end{proof}

\begin{corollary}[{\cite{cone,scla}}]\label{cor.Linfinitosultot}
There exists a canonical $L_\infty$-algebra structure
$\widetilde{\tot}(\mathfrak{g}^{\Delta})$ on the differential
graded vector space ${\tot}(\mathfrak{g}^{\Delta})$, together with
an injective quasi-isomorphism $E_{\infty}\colon
\widetilde{\tot}(\mathfrak{g}^{\Delta})\to
{\tot}_{TW}(\mathfrak{g}^\Delta)$.
\end{corollary}

\begin{proof}
It is sufficient to apply Theorem~\ref{teo trasfer strutt L
infiniti} to the morphisms $I,E,h$, in order to define a canonical
$L_\infty$-algebra structure
$\widetilde{\tot}(\mathfrak{g}^{\Delta})$ on
${\tot}(\mathfrak{g}^{\Delta})$, and an injective
quasi-isomorphism $E_{\infty}\colon
\widetilde{\tot}(\mathfrak{g}^{\Delta})\to
{\tot}_{TW}(\mathfrak{g}^\Delta)$ extending $E$. \par

Notice that $E_{\infty}$ induces an isomorphism of functors
$\Def_{\widetilde{\tot}(\mathfrak{g}^{\Delta})}\xrightarrow{\;
\simeq\;} \Def_{{\tot}_{TW}(\mathfrak{g}^{\Delta})}$.
\end{proof}

Let $\chi\colon L\to M$ be a morphism of differential graded Lie
algebras over a field $\K$ of characteristic $0$. We have already
seen, in Example~\ref{ex.conisonoscla},  that $\chi$ can be
interpreted as the semicosimplicial DGLA

 \[\mathfrak{\chi}^\Delta:\quad
\xymatrix{ {L}
\ar@<2pt>[r]^0\ar@<-2pt>[r]_{\chi} & { M}
      \ar@<4pt>[r] \ar[r] \ar@<-4pt>[r] & 0\cdots},\]

and therefore we have a canonical $L_{\infty}$ structure on the
total  complex
\[ \tot(\chi^\Delta)=\mathop{\oplus}_i\tot(\chi^\Delta)^i,
\qquad \tot(\chi^\Delta)^i=
L^i\oplus M^{i-1}.\]%
The brackets
\[
\mu_n\colon\bigwedge^n\tot(\chi^\Delta)\to
\tot(\chi^\Delta)[2-n],\qquad n\ge 1,
\]
have been explicitly described in
\cite{cone}. Namely, one has
\[ \mu_1(l,m)=(dl,\chi(l)-dm),\qquad l\in L, m\in M,\]%
\[ \mu_2((l_1,m_1)\wedge (l_2,m_2))=
\left([l_1,l_2],\frac{1}{2}[m_1,\chi(l_2)]+
\frac{(-1)^{\deg(l_1)}}{2}[\chi(l_1),m_2]\right)\]
and for $n\ge 3$
\[
\mu_n((l_1,m_1)\wedge\cdots\wedge(l_n,m_n))=\pm
\frac{B_{n-1}}{(n-1)!}\sum_{\sigma\in S_n}\varepsilon(\sigma)
[m_{\sigma(1)},[\cdots,[m_{\sigma(n-1)},\chi(l_{\sigma(n)})]\cdots]].
\]
Here the $B_n$'s are the Bernoulli numbers, $\varepsilon$ is
the Koszul sign and we refer to
\cite{cone} for the exact determination of the overall
$\pm$ sign in the  formulas (it will not be needed in
the present paper).

\begin{proposition}\label{prop.criterion}
In the notation above, assume that: \begin{enumerate}

\item $\chi\colon L\to M$ is injective,

\item $\chi\colon H^*(L)\to H^*(M)$ is injective.
\end{enumerate}
Then,   the $L_{\infty}$-algebra $\widetilde{\tot}(\chi^{\Delta})$
is quasi-abelian.
\end{proposition}

\begin{proof}
Since $H^0(\Hom^*(L,L))=\Hom^0(H^*(L),H^*(L))$ and
$H^0(\Hom^*(M,L))=\Hom^0(H^*(M),H^*(L))$, the surjective morphism
\[ \Hom^*(M,L)\to \Hom^*(L,L),\qquad \phi \mapsto \phi \chi, \]
induces  surjective maps
\[
H^0(\Hom^*(M,L))\to H^0(\Hom^*(L,L)),\qquad
Z^0(\Hom^*(M,L))\to Z^0(\Hom^*(L,L))
\]
and then the identity on  $L$ can be lifted to a morphism
$\pi\colon M\to L$ of differential graded vector spaces. Denoting
$V=\ker(\pi)$, we have a direct sum decomposition of differential
graded vector spaces $M=\chi(L)\oplus V$.

Denote by $d$ the differential on $M$, by assumption $d(V)\subset
V$ and the inclusion $V[-1]\hookrightarrow \tot(\chi^{\Delta})$ is
an injective quasi-isomorphism. Let $H$ be a graded vector space,
endowed with trivial differential, and $g\colon H\to V[-1]$  an
injective morphism of differential graded vector spaces, inducing
an isomorphism $H\mapor{\sim}H^*(V[-1])$. The linear map
\[
f\colon H\to \widetilde{\tot}(\chi^{\Delta})),\qquad f(h)=(0,g(h)),
\]
is annihilated by every bracket of
$\widetilde{\tot}(\chi^{\Delta})$  and then it is an injective
quasi-isomorphism of $L_{\infty}$-algebras.
\end{proof}

\begin{example}\label{exe. marco Hom(U,U) Hom(WW)}
Let $W$ be a differential graded vector space and let $U\subset W$
be a differential graded subspace. Assume that the induced
morphism $H^*(U)\to H^*(W)$ is  injective, then the injective
morphism of DGLAs
\[
\chi\colon\{f\in\Hom^*(W,W)\mid f(U)\subset U\}\to \Hom^*(W,W)
\]
satisfies the hypothesis of Proposition~\ref{prop.criterion}. In
fact, the same argument used above shows that there exists a
direct sum decomposition of differential graded vector spaces
$W=U\oplus V$. Next, consider   the subspace
\[
K=\{f\in \Hom^*(W,W)\mid f(W)\subset V,\; f(V)=0\}.
\]
It is straightforward to check that $K$ is a complementary
subcomplex of the image of $\chi$ inside $\Hom^*(W,W)$.
\end{example}

\begin{remark} Let $W$ be a differential graded vector space and let
$U\subset W$ be a differential graded subspace. It is showed in
\cite{Periods},  that the deformation functor associated with the
morphism of DGLAs
\[
\chi\colon\{f\in\Hom^*(W,W)\mid f(U)\subset U\}\to \Hom^*(W,W),
\]
i.e., the deformation functor associated to the
$L_{\infty}$-algebra $\widetilde{\tot}(\chi^{\Delta})$, has a
natural interpretation as the local structure of the derived
Grassmannian of $W$ at the point $U$. Moreover,
Proposition~\ref{prop.criterion} implies that the derived
Grassmannian of $W$ is smooth at the points corresponding to
subspaces $U$ such that $H^*(U)\to H^*(W)$ is  injective.
\end{remark}

\section{Semicosimplicial Cartan homotopies}\label{sez.cartan
homotopy}

The abstract notion of Cartan homotopy has been introduced  in
\cite{Periods,Generalized fi-ma} as a powerful tool for the construction
of $L_{\infty}$ morphisms.

\begin{definition}\label{def.cartanhomotopy}
Let $L$ and $M$ be two differential graded Lie algebras. A linear map of degree $-1$
\[ \bi\colon L \to M  \]
is called a \emph{Cartan homotopy} if, for every $a,b\in L$, we
have:
\[\bi_{[a,b]}=[\bi_a,d_M\bi_b+\bi_{d_L b}] \qquad \text{and }
 \qquad [\bi_a,\bi_{b}]=0.\]
\end{definition}

For every Cartan homotopy $\bi$, it is convenient to consider the
map
\[ \bl\colon L\to M,\qquad
\bl_a=d_M\bi_a+\bi_{d_L a}.
\]
It is straightforward to check that $\bl$ is a morphism of DGLAs
and the conditions of Definition~\ref{def.cartanhomotopy} become
\[\bi_{[a,b]}=[\bi_a,\bl_b] \qquad \text{and } \qquad [\bi_a,\bi_{b}]=0.\]
As a morphism of complexes, $\bl$ is homotopic to 0 (with homotopy
$\bi$).

\begin{example}\label{exam.cartan su ogni aperto}
Let $X$ be a smooth algebraic  variety, $\Theta_X$ the tangent
sheaf and  $(\Omega^{\ast}_X,d)$ the algebraic de Rham complex.
Then, for every open subset $U\subset X$,  the contraction
\[ \Theta_X(U)\otimes\Omega^k_X(U)\xrightarrow{\quad\contr\quad}
\Omega^{k-1}_X(U)\]
induces a linear map of degree $-1$
\[
\bi\colon \Theta_X(U)\to \Hom^*(\Omega^{*}_X(U),
\Omega^{*}_X(U)),\qquad \bi_{\xi}(\omega)=\xi\contr\omega
\]
that  is a Cartan homotopy. In fact, the differential of
$\Theta_X(U)$ is trivial and the differential on the DGLA
$\Hom^{*}(\Omega^{*}_X(U), \Omega^{*}_X(U))$ is $\phi\mapsto
[d,\phi]=d\phi-(-1)^{\deg(\phi)}\phi d$. Therefore, since $\bi_a$
has degree $-1$, we have that $\bl_a=d \bi_a+\bi_ad$ is the Lie
derivative and the above conditions reduce to the classical
Cartan's homotopy formulas:
\begin{enumerate}
\item  $[\bi_a,\bi_b]=0$;
\item $\bi_{[a,b]}=\bl_a \bi_b-\bi_b\bl_a=[\bl_a,\bi_b]=[\bi_a,\bl_b]$,
\end{enumerate}
\end{example}

The relation between Cartan's homotopy and $L_\infty$ algebras is
given by the following theorem.

\begin{theorem}[{\cite[Corollary~3.7]{Generalized fi-ma}}]
\label{thm.FM cartan --- morfismo} Let $\chi\colon N\to M$ be a
morphism of DGLAs and $\bi \colon L\to M$ be a Cartan
homotopy. Assume that $\phi\colon L\to N$ is a morphism of DGLAs,
such that $\chi\phi=\bl=d_M\bi+\bi d_L$. Then, the linear map
\[ \Phi \colon L \longrightarrow \widetilde{\tot}(\chi^{\Delta}) \qquad
\Phi(a)=(\phi(a),\bi_a)
\]
is a linear $L_\infty$-morphism. In particular, if $N$ is a
subalgebra of $M$ containing  $\bl(L)$ and $\chi$ is the
inclusion, then the map
\[ L \longrightarrow \widetilde{\tot}(\chi^{\Delta}) \qquad
a\mapsto (\bl_a,\bi_a)
\]
is a linear $L_\infty$-morphism.
\end{theorem}

\begin{proof}
Straightforward consequence of the explicit description of the
$L_{\infty}$ structure of $\widetilde{\tot}(\chi^{\Delta})$.
\end{proof}

\begin{remark}\label{oss homotopy dopo morphi is homotopy}
It is plain from definition that Cartan homotopies are stable
under composition with morphisms of DGLAs. More precisely, if
$f\colon L'\to L$ and $g\colon M\to M'$ are morphisms of
differential graded Lie algebras and $\bi\colon L\to M$ is a
Cartan homotopy, then also $g\bi f\colon L'\to M'$ is a  Cartan
homotopy.
\end{remark}

\begin{lemma}\label{lem homotopy implica homotopyX A}
Let $\bi\colon L\to M$ be a Cartan homotopy and $A$ be a
differential graded commutative algebra.  Then the map
\begin{align*}
\bi\otimes \Id\colon & L\otimes A\to M\otimes A\\
 &(x \otimes a)\mapsto\bi_x \otimes a
\end{align*}
is a Cartan homotopy.
\end{lemma}

\begin{proof}
By definition, for any $ x\otimes a$ and $y\otimes b \in L\otimes
A$, we have
\[
[(\bi\otimes \Id)_{x\otimes a},(\bi\otimes \Id)_{y\otimes b}]=
[\bi_x \otimes a ,\bi_y \otimes b ]=(-1)^{\deg(a)
(\deg(y)-1)}[\bi_x,\bi_y ]\otimes ab=0.
\]
Moreover, denoting $\bl=d_M\bi+\bi d_L$ and
$\tilde{\bl}=d_{M\otimes A}(\bi\otimes Id) +(\bi\otimes Id)
d_{L\otimes A}$ we have
\[ \tilde{\bl}_{x\otimes a}=
d_{M\otimes A}(\bi_x\otimes a)+(\bi\otimes Id)(d_{L}x\otimes
a+(-1)^{\deg(x)}x\otimes d_Aa)= \]
\[= d_M\bi_x\otimes a-(-1)^{\deg(x)}x\otimes
 d_Aa+\bi_{d_Lx}\otimes a +(-1)^{\deg(x)}x\otimes d_Aa
=\bl_{x}\otimes a.\] Thus, for any $ x\otimes a$ and $y\otimes b
\in L\otimes A$, we get
\[
(\bi\otimes \Id)_{[x\otimes a,y\otimes b]}=
(\bi\otimes \Id)_{(-1)^{\deg(a)\deg(y)}[x,y]\otimes ab}=
(-1)^{\deg(a)\deg(y)}\bi_{[x,y]}
\otimes ab=\]
\[
(-1)^{\deg(a)\deg(y)}[\bi_x,\bl_y]\otimes ab=[\bi_x\otimes
a,\bl_y\otimes b]=
[(\bi\otimes \Id)_{x\otimes a},\tilde{\bl}_{y\otimes b}].
\]
\end{proof}

\begin{definition}\label{def.contraction}
 Let $L$ be a differential graded Lie algebra and $V$ a differential
 graded vector space. A bilinear map
\[ L\times V\xrightarrow{\quad\contr\quad} V\]
of degree $-1$ is called a \emph{contraction} if the induced map
\[ \bi\colon L\to \Hom^*(V,V),\qquad \bi_l(v)=l\contr v,\]
is a Cartan Homotopy.\end{definition}

The notion of contraction is stable under scalar extensions, more
precisely:

\begin{lemma}\label{lem. cartan implica cartan XA}

Let $V$ be  a differential graded vector space  and
\[
L\times V\xrightarrow{\quad\contr\quad} V
\]
a contraction. Then, for every differential graded commutative
algebra $A$, the natural extension
\[
(L\otimes A)\times (V\otimes A) \xrightarrow{\quad\contr\quad}
(V\otimes A) \qquad (l\otimes a)\contr (v\otimes
b)=(-1)^{\deg(a)\deg(v)}l \contr v\otimes ab,
\]
is a contraction.



\end{lemma}

\begin{proof}
According to Remark~\ref{oss homotopy dopo morphi is homotopy},
Lemma~\ref{lem homotopy implica homotopyX A} and Definition
\ref{def.contraction}, it is sufficient to prove that  the natural
map
\[ \alpha\colon \Hom^*(V,V)\otimes A\to \Hom^*(V\otimes A,V\otimes A),
\quad \alpha(\phi \otimes a)(v \otimes
b)=(-1)^{\deg(a)\deg(v)}\phi(v) \otimes ab,\] is a morphism of
DGLAs. This is completely straightforward  and it is left to the
reader.
\end{proof}

The notions of Cartan homotopy and contraction extend naturally to
the semicosimplicial setting. Here, we consider only the case of
contractions.

\begin{definition}\label{def.contractioncosimpl}
Let $\mathfrak{g}^\Delta$ be a semicosimplicial DGLA and
$V^\Delta$ a semicosimplicial differential graded vector space.
A semicosimplicial contraction
\[ \mathfrak{g}^\Delta\times V^\Delta\xrightarrow{\;\contr\;} V^\Delta,\]
is a sequence of contractions $\mathfrak{g}_n\times V_n\xrightarrow{\;
 \contr\;} V_n$, $n\ge 0$,
commuting with coface maps, i.e., $\de_k(l\contr v)=\de_k(l)\contr
\de_k(v)$, for every $k$.
\end{definition}

\begin{proposition}\label{prop.TWforcontractions}
Every semicosimplicial contraction
\[ \mathfrak{g}^\Delta\times V^\Delta\xrightarrow{\;\contr\;} V^\Delta\]
extends  naturally to a contraction
\[ \tot_{TW}(\mathfrak{g}^\Delta)\times \tot_{TW}(V^\Delta)
\xrightarrow{\;\contr\;} \tot_{TW}(V^\Delta).\]
\end{proposition}

\begin{proof}
By definition, for every $n$, we have a Cartan homotopy
\[
\bi: \mathfrak{g}_n \to \Hom^*(V_n , V_n),
\]
and so, by  Lemma~\ref{lem. cartan implica cartan XA}, a Cartan
homotopy
\[
\bi\colon \mathfrak{g}_n \otimes (A_{PL})_n \to \Hom^*(V_n\otimes
(A_{PL})_n ,V_n\otimes (A_{PL})_n).
 \]
Therefore, it is enough to prove that $\bi_{x}(y) \in
\tot_{TW}(V^\Delta)$, for every pair of  sequences
$x=(x_n)_{n\in\mathbb{N}} \in \tot_{TW}({\mathfrak g}^\Delta)$ and
$y=(y_n)_{n\in\mathbb{N}} \in \tot_{TW}(V^\Delta)$; it follows
from Remark \ref{remark. tot X tot -- tot}.

\end{proof}

\bigskip
\section{Semicosimplicial Lie algebras  and deformations
of smooth varieties}\label{sezione.deformazioni varieta}

Let ${\g}^\Delta$ be a  semicosimplicial Lie algebra, then we can
apply the construction of Section~\ref{sez.formality map-cone} in
order to construct the Thom-Whitney DGLA
${\tot}_{TW}(\g^{\Delta})$, the $L_{\infty}$-algebra
$\widetilde{\tot}(\g^{\Delta})$ and their associated (and
isomorphic) deformation functors
$\Def_{\widetilde{\tot}(\g^{\Delta})}\simeq
\Def_{{\tot}_{TW}(\g^{\Delta})}$.

Beyond this way, there is another natural, and more geometric, way
to  define a deformation functor, see \cite[Definitions~1.4 and
1.6]{Pridham} and \cite[Section~3]{scla}. More precisely, if ${
\g}^\Delta$ is a semicosimplicial Lie algebra, then we denote
\[
 Z^1_{sc}( \exp { \g}^\Delta ) \colon \Art \to \Set
\]
  as
\[
Z ^1_{sc}(\exp { \g}^\Delta )(A)=\{ x \in {\g}_1 \otimes
\mathfrak{m}_A   \ |\  e^{\de_{0}(x)}e^{-\de_{1}(x)}e^{\de_{2}(x)}=1
\},
\]
and
\[
 H^1_{sc}( \exp { \g}^\Delta) \colon \Art \to \Set
\]
such that
\[
H^1_{sc}( \exp { \g}^\Delta)(A)= Z^1_{sc}(\exp { \g}^\Delta)(A)/ \sim,
\]
where $x \sim y$ if and only if there exists
$a\in {\g}_0\otimes\mathfrak{m}_A$, such that
$e^{-\de_{1}(a)}e^{x}e^{\de_{0}(a)}=e^y$.

\begin{example} Let $\mathcal L$ be a sheaf of Lie
algebras on a paracompact topological space $X$, and $\mathcal U$
an open covering of $X$; it is naturally defined the \v{C}ech
semicosimplicial Lie algebra $\mathcal L(\mathcal U)$
\[ \mathcal{L}(\mathcal{U}):\qquad
\xymatrix{ {\prod_i\mathcal{L}(U_i)}
\ar@<2pt>[r]\ar@<-2pt>[r] & {
\prod_{i<j}\mathcal{L}(U_{ij})}
      \ar@<4pt>[r] \ar[r] \ar@<-4pt>[r] &
      {\prod_{i<j<k}\mathcal{L}(U_{ijk})}
\ar@<6pt>[r] \ar@<2pt>[r] \ar@<-2pt>[r] \ar@<-6pt>[r]& \cdots},
\]
and, for every $A\in\mathbf{Art}$, the set $H^1_{sc}( \exp
{\mathcal{L}(\mathcal{U})})(A)$ is exactly the cohomology set
$H^1(\mathcal{U},\exp(\mathcal{L}\otimes\mathfrak{m}_A))$
\cite{hirzebruch}.
\end{example}

The relation between the above functors is given by the  following theorem.

\begin{theorem}\label{teor.H^1=Def TOT}
Let  ${ \g}^\Delta$ be  a semicosimplicial Lie algebra. Then, for
every $A\in\mathbf{Art}$,
\[
\MC_{\widetilde{\tot}(\g^\Delta)}(A) =Z^1_{sc}(\exp\g^\Delta)(A),
\]
as subsets of ${\g}_1 \otimes \mathfrak{m}_A$. Moreover, we have
natural isomorphisms of deformation functors
\[
\Def_{{\tot}_{TW}(\g^{\Delta})} \simeq \Def_{\widetilde{\tot}(\g^\Delta)}
\xrightarrow{\sim} H^1_{sc}(\exp\g^\Delta),
\]
\end{theorem}
\begin{proof} For the proof, we refer to
\cite{scla}.
\end{proof}

Next, assume   that $X$ is a smooth  algebraic variety over a
field $\K$ of characteristic $0$, with tangent sheaf $\Theta_X$,
and let $\mathcal{U}=\{U_i\}_{i \in I}$ be an  affine open
covering of $X$.

Since every infinitesimal deformation of a
smooth affine scheme is trivial \cite[Lemma II.1.3]{Sernesi}, every
infinitesimal deformation  $X_A$ of $X$ over $\Spec(A)$ is
obtained by
gluing the trivial deformations $U_i \times \Spec (A)$ along  the
double intersections $U_{ij}$, and therefore it is determined by
the sequence  $\{\theta_{ij}\}_{i<j}$ of   automorphisms of sheaves of $A$-algebras
\[\xymatrix{ &\mathcal{O}(U_{ij}) & \\
\mathcal{O}(U_{ij})\otimes A \ar[rr]^{\theta_{ij}}_{\simeq}\ar[ur] & &
 \mathcal{O}(U_{ij})\otimes A\ar[ul] \\
           & A\ar[ru]\ar[lu] &   \\ }\]
satisfying the cocycle condition
\begin{equation}\label{equa.cociclo auto Uij}
 \theta _{jk}  \theta _{ik}^{-1}
   \theta _{ij}=\Id_{\mathcal{O}(U_{ijk})\otimes A}, \qquad
   \forall \ i<j<k \in I.
\end{equation}
Since we are in characteristic zero, we  can take the logarithms
and write $\theta_{ij}=e^{d_{ij}}$,  where
$d_{ij}\in\Theta_X(U_{ij})\otimes\mathfrak{m}_A$. Therefore, the
Equation (\ref{equa.cociclo auto Uij}) is equivalent to
\[
e^{d_{jk}}e^{-d_{ik}}e^{d_{ij}}=1 \in
 \exp(\Theta_X(U_{ijk})\otimes\mathfrak{m}_A),
\qquad \forall \ i<j<k \in I.
\]

Next, let $X_A'$ be another deformation of $X$ over $\Spec(A)$,
defined by the  cocycle $\theta_{ij}'$. To give  an isomorphism of
deformations  $X_A' \simeq X_A $ is the same to give, for every
$i$, an automorphism $\alpha_i$ of $\mathcal{O}(U_i)\otimes A$
such that $\theta_{ij}= {\alpha_i}^{-1}
{\theta_{ij}'}^{-1}\alpha_j$, for every $i<j$. Taking again
logarithms, we can write $\alpha_i=e^{a_i}$, with $a_i \in
\Theta_X(U_i)\otimes\mathfrak{m}_A$,  and so
$e^{-a_i}e^{d'_{ij}}e^{a_j}=e^{d_{ij}}$.

\begin{theorem}\label{thm.defoXcomedeftot}
Let $\mathcal{U}$ be an affine open cover of a smooth algebraic
variety  $X$ defined over  an algebraically closed  field of
characteristic 0. Denoting by  $\Def_X$ the functor of
infinitesimal deformations of $X$,  there exist isomorphisms of
functors
 \[
\Def_X \cong H^1 _{sc}( \exp \Theta_X(\mathcal{U}))\cong
\Def_{\tot_{TW}(\Theta_X(\mathcal{U}))}
 \cong
\Def_{\widetilde{\tot}(\Theta_X(\mathcal{U}))}\, ,
\]
where $\Theta_X(\mathcal{U})$  is the
semicosimplicial Lie algebra defined  in
Example~\ref{example. chec Tangen X}.
\end{theorem}

\begin{proof} By Theorem~\ref{teor.H^1=Def TOT}, it is  sufficient to prove
$\Def_X \cong H^1 _{sc}( \exp \Theta_X(\mathcal{U}))$. By
definition,
\[
Z ^1_{sc}(\Theta_X(\mathcal{U}) )(A)=\{ \{x_{ij}\} \in {
\prod_{i<j}\Theta_X(U_{ij})} \otimes \mathfrak{m}_A   \mid
e^{x_{jk}}e^{-x_{ik}}e^{x_{ij }}=1 \  \forall \ i<j<k\},
\]
for each $A \in \Art$. Moreover, given  $x=\{x_{ij}\} $ and
$y=\{y_{ij}\} \in  { \prod_{i<j}\Theta_X(U_{ij})} \otimes
\mathfrak{m}_A $, we have $x\sim y$ if and only if there exists
$a=\{a_{i}\} \in { \prod_{i}\Theta_X(U_{i})} \otimes
\mathfrak{m}_A $ such  that   $e^{-a_{j }}e^{ x_{ij}}e^{a_{i
}}=e^{y_{ij}}$ for all $i<j$.
\end{proof}

\begin{remark} \label{remark.ko-spen-alge =tot}
Note that, when  $\K=\mathbb{C}$, the $L_{\infty}$-algebra $
\widetilde{\tot} (\Theta_X(\mathcal{U}))$ is quasi-isomorphic to
the Kodaira-Spencer differential graded Lie algebra of $X$.

\end{remark}

\section{Proof of the main theorem}\label{sez.Proof main theorem}

In this section, we use the results developed before to give  a
complete  algebraic proof of the following  theorem.

\begin{theorem}\label{thm.maintheorem}
Let  $X$ be a smooth projective variety  of dimension $n$, defined
over an algebraically closed field of characteristic 0. If the
contraction map
\[
H^*(\Theta_X)\xrightarrow{\bi}\Hom^*(H^*(\Omega^n_X),H^*(\Omega^{n-1}_X))\]
is injective, then, for every affine open cover $\mathcal{U}$ of
$X$, the DGLA $\tot_{TW}(\Theta_X(\mathcal{U}))$ is quasi-abelian.
\end{theorem}

\begin{proof}
According to  Lemma~\ref{lem.criterioquasiabelianita}, it is
sufficient to prove that there exist a quasi abelian
$L_{\infty}$-algebra $H$ and a morphism
${\tot}_{TW}(\Theta_X(\mathcal{U}))
 \to H$  that is injective in cohomology.

Let $n$ be the dimension of $X$  and denote by
$\Omega^{\ast}_X$ the algebraic de Rham complex.
For every $i\le n$, let $\check{C}(\mathcal{U},\Omega^i_X)$ be
 the \v{C}ech complex of the coherent sheaf $\Omega_X^i$, with
respect to the affine cover $\mathcal{U}$, and
$\check{C}(\mathcal{U},\Omega^{\ast}_X)$  the total complex of the
semicosimplicial differential graded vector space
$\Omega^{\ast}_X(\mathcal{U})$ (Example~\ref{example. chec OMEGA
X}). Notice that
\[
\check{C}(\mathcal{U},\Omega^{\ast}_X)^i=
\bigoplus_{a+b=i}\check{C}(\mathcal{U},\Omega^{a}_X)^b.
\]
and $\check{C}(\mathcal{U},\Omega^{n}_X)$ is a subcomplex of
$\check{C}(\mathcal{U},\Omega^{\ast}_X)$.

Then, we have a commutative diagram of complexes with horizontal
quasi-isomorphisms:

\[\xymatrix{ \check{C}(\mathcal{U},\Omega^{n}_X)=\tot( \Omega^{n}_X(\mathcal{U}))
\ar[rr]^{E} \ar@{^{(}->}[d] & & {\tot}_{TW}( \Omega^{n}_X(\mathcal{U}))
 \ar@{^{(}->}[d] \\  \check{C}(\mathcal{U},\Omega^{*}_X)=\tot( \Omega^{*}_X(\mathcal{U}))
\ar[rr]^{E} & & {\tot}_{TW}( \Omega^{*}_X(\mathcal{U})).
}\]

Since $\K$ has characteristic 0 and $X$ is smooth and  proper, the
Hodge spectral sequence degenerates at $E_1$ (we refer to
\cite{faltings,DI} for a purely algebraic proof of this fact).
Therefore, we have  injective maps
\[ H^*(X,\Omega^n_X)=H^*(\check{C}(\mathcal{U},\Omega^{n}_X))
\hookrightarrow
H^*(\check{C}(\mathcal{U},\Omega^{\ast}_X))=H^*_{DR}(X/\K).\]
\[ H^*(X,\Omega^{n-1}_X)=H^*(\check{C}(\mathcal{U},\Omega^{n-1}_X))
\hookrightarrow
H^*\left(\frac{\check{C}(\mathcal{U},\Omega^{\ast}_X)}{
\check{C}(\mathcal{U},\Omega^{n}_X)}\right).\]

Thus, the natural inclusions of complexes
\[{\tot}_{TW}(\Omega^{n}_X(\mathcal{U}))\to {\tot}_{TW}(
\Omega^{*}_X(\mathcal{U})),\]
\[{\tot}_{TW}(\Omega^{n-1}_X(\mathcal{U}))\to \frac{\tot_{TW}(
\Omega^{*}_X(\mathcal{U}))}{{\tot}_{TW}(\Omega^{n}_X(\mathcal{U}))},\]
are  injective in cohomology.

According to Example~\ref{exam.cartan su ogni aperto},
for every open subset $U\subset X$,
the contraction
of vector fields with
differential forms   defines a Cartan homotopy:
\[
\bi\colon \Theta_X(U)\to \Hom^{*}(\Omega^{*}_X(U),
\Omega^{*}_X(U)),\qquad \bi_{\xi}(\omega)=\xi\contr\omega.
\]
Since the contraction ${\mspace{1mu}\lrcorner\mspace{1.5mu}}$
commutes with   restrictions to open subsets, we have a
semicosimplicial contraction
\[
\Theta_X(\mathcal{U})\times \Omega^{*}_X(\mathcal{U})
\xrightarrow{\;\contr\;} \Omega^{*}_X(\mathcal{U}),
\]
and, by  Proposition~\ref{prop.TWforcontractions}, this induces
naturally a Cartan homotopy
\[\bi\colon {\tot}_{TW}( \Theta_X(\mathcal{U})) \longrightarrow
\Hom^*({\tot}_{TW}( \Omega^*(\mathcal{U})),
{\tot}_{TW}(\Omega^*(\mathcal{U}))).\]

Notice that, for every $\xi\in {\tot}_{TW}(\Theta_X(\mathcal{U}))$
and every $i$, we have
\[ \bi_{\xi}({\tot}_{TW}(\Omega^i(\mathcal{U})))\subset
{\tot}_{TW}(\Omega^{i-1}(\mathcal{U})),\]
\[ \bl_{\xi}({\tot}_{TW}(\Omega^i(\mathcal{U})))\subset
{\tot}_{TW}(\Omega^{i}(\mathcal{U})),\qquad \bl_{\xi}
=d\bi_{\xi}+\bi_{d\xi}.\]

Moreover, the assumption of the theorem, together with
\cite[3.1]{navarro}, implies that the map \[{\tot}_{TW}(
\Theta_X(\mathcal{U})) \xrightarrow{\bi}
\Hom^*\left(\tot_{TW}(\Omega^{n}_X(\mathcal{U})),
\tot_{TW}(\Omega^{n-1}_X(\mathcal{U}))\right)\] is injective in
cohomology.\par

Next, consider the differential graded Lie algebras
\[
M=\Hom^*({\tot}_{TW}( \Omega^{*}_X(\mathcal{U})), {\tot}_{TW}
\Omega^{*}_X(\mathcal{U})),
\]
\[
L=\{f\in M \mid
f({\tot}_{TW}( \Omega^{n}_X(\mathcal{U})))\subset
{\tot}_{TW}( \Omega^{n}_X(\mathcal{U}))\},
\]
and let $\chi\colon L \to M$ be the inclusion. According to
Example~\ref{exe. marco Hom(U,U) Hom(WW)}, the
$L_{\infty}$-algebras $\widetilde{\tot}(\chi^{\Delta})$ is
quasi-abelian.

Moreover, we have
$\bl({\tot}_{TW}(\Theta_X(\mathcal{U})))\subset L$ and so, by
Theorem \ref{thm.FM cartan --- morfismo}, there exists a linear
$L_\infty$-morphism
\[
{\tot}_{TW}( \Theta_X(\mathcal{U})) \xrightarrow{(\bl,\bi)}
\widetilde{\tot}(\chi^{\Delta}),\quad
x \mapsto (\bl_x, \bi_x).
\]
Since the map $\chi$ in injective, its mapping cone
${\tot}(\chi^{\Delta})$ is quasi-isomorphic to its  cokernel
\[\coker \chi= \Hom^*\left(\tot_{TW}(\Omega^{n}_X(\mathcal{U})),
\frac{\tot_{TW}(\Omega^{*}_X(\mathcal{U}))}{
\tot_{TW}(\Omega^{n}_X(\mathcal{U}))}\right)\]
and we have a commutative diagram of complexes

\[\xymatrix{ {\tot}_{TW}( \Theta_X(\mathcal{U}))
\ar[rr]^{(\bl,\bi)} \ar[d]^{\bi} & &  {\tot}(\chi^{\Delta})
 \ar[d]^{q-iso} \\
\Hom^*\left(\tot_{TW}(\Omega^{n}_X(\mathcal{U})),
\tot_{TW}(\Omega^{n-1}_X(\mathcal{U}))\right)
\ar[rr]^{\qquad\qquad\qquad\alpha\!\!\!\!} & & \coker \chi.
\\ }\]

Since both $\bi$ and $\alpha$ are injective in cohomology, also
the $L_\infty$-morphism $(\bl,\bi)$ is injective in cohomology.

\end{proof}

\begin{theorem}\label{cor.trivial K- TOT quasiabelia}
Let  $\mathcal{U}=\{U_i\}$ be  an affine open  cover of a smooth
projective variety  $X$ defined over an algebraically closed field
of characteristic 0. If the canonical bundle of $X$ is trivial or
torsion, then  the DGLA $\tot_{TW}(\Theta_X(\mathcal{U}))$ is
quasi-abelian.
\end{theorem}

\begin{proof} Assume first that $X$ has trivial canonical bundle.
If $n$ is the dimension of $X$, the cup product with a nontrivial
section of the canonical bundle gives the isomorphisms
$H^i(\Theta_X)\simeq H^i(\Omega^{n-1}_X)$ and the conclusion
follows immediately from Theorem~\ref{thm.maintheorem}. If $X$ has
torsion canonical bundle we may consider the canonical cyclic
cover $\pi\colon Y\to X$ and the affine open cover
$\mathcal{V}=\{\pi^{-1}(U_i)\}$. Now the variety $Y$ has trivial
canonical bundle and then the $L_{\infty}$-algebra
$\widetilde{\tot}(\Theta_Y(\mathcal{V}))$ is quasi-abelian. Since
$\pi$ is an unramified cover the natural injective map
$\widetilde{\tot}(\Theta_X(\mathcal{U}))\to\widetilde{\tot}(\Theta_Y(\mathcal{V}))$
is also injective in cohomology and we conclude the proof by using
the same argument of Theorem~\ref{thm.maintheorem}.
\end{proof}

\begin{remark}
When $\K=\mathbb{C}$, the previous theorems together with Remark
\ref{remark.ko-spen-alge =tot}, implies that the Kodaira-Spencer
DGLA is quasi abelian, for a projective manifold with trivial or
torsion  canonical bundle.

\end{remark}

\begin{theorem}\label{thm.kodairaprinciple}
Let  $X$ be a smooth projective variety  of dimension $n$ defined
over an algebraically closed field of characteristic 0. Then, the
obstructions to deformations of $X$ are contained in the kernel of
 the contraction map
\[ H^2(\Theta_X)\xrightarrow{\bi}\prod_{p}
\Hom(H^p(\Omega^n_X),H^{p+2}(\Omega^{n-1}_X)).\]
\end{theorem}

\begin{proof}  We have seen in the proof of Theorem~\ref{thm.maintheorem}
that, for every affine open cover $\mathcal{U}$ of $X$, there
exists an $L_{\infty}$-morphism
$\widetilde{\tot}(\Theta_X(\mathcal{U}))\to
\widetilde{\tot}(\chi^{\Delta})$ and that
$\widetilde{\tot}(\chi^{\Delta})$ is quasi-abelian. Therefore, we
are in the condition to apply the general strategy used in
\cite{CCK,ManettiSeattle,IaconoSemireg}: the deformation functor
associated to $\widetilde{\tot}(\chi^{\Delta})$ is unobstructed
and the obstructions of
$\Def_X\simeq\Def_{\widetilde{\tot}(\Theta_X(\mathcal{U}))}$ are
contained in the kernel of the obstruction map
$H^2(\tot(\Theta_X(\mathcal{U})))\to H^2({\tot}(\chi^{\Delta}))$.
\end{proof}

\begin{corollary}\label{cor.btt}
Let $X$ be   a smooth projective variety defined over an
algebraically closed field of characteristic 0. If the canonical
bundle of $X$ is trivial, then $X$ has unobstructed deformations.
\end{corollary}

\begin{proof}
The previous Corollary~\ref{cor.trivial K- TOT quasiabelia} implies that
$\widetilde{\tot}(\Theta_X(\mathcal{U}))$ is quasi-abelian and so
$\Def_{\widetilde{\tot}(\Theta_X(\mathcal{U}))}$  is smooth. By
Theorem~\ref{thm.defoXcomedeftot}, $\Def_X \cong
\Def_{\widetilde{\tot}(\Theta_X(\mathcal{U}))}$.
\end{proof}

\begin{remark} Transcendental proofs of the analogue of
Theorem~\ref{thm.kodairaprinciple} for compact K\"{a}hler
manifolds can be found in \cite{CCK,clemens99,ManettiSeattle},
while we refer to \cite{IaconoSemireg,ManettiSeattle} for the
proof that the $T^1$-lifting is definitely insufficient for
proving Theorem~\ref{thm.kodairaprinciple}.
\end{remark}

\end{document}